\documentclass[12pt,a4paper]{amsart}         
\usepackage[margin=3cm]{geometry}
\usepackage[backend=bibtex,url=true,doi=true,isbn=false,giveninits=true,style=numeric,sorting=nyt,eprint=false]{biblatex}
\usepackage[T1]{fontenc}
\usepackage{amsmath,amsfonts,amssymb,amsthm}
\usepackage{mathrsfs,epsfig,epstopdf,url}
\usepackage{enumerate}
\usepackage{xcolor}
\usepackage{float}
\usepackage{tikz}
\restylefloat{table}
\newcommand{\C}{C_L(D,sP_{\infty})\mid_{\mathbb{F}_r}}
\theoremstyle{definition}
\newtheorem{definition}{Definition}%

\theoremstyle{plain}
\newtheorem{theorem}[definition]{Theorem}

\newtheorem{proposition}[definition]{Proposition}

\newtheorem{remark}[definition]{Remark}

\title{On the nonexistence of certain orthogonal arrays of strength four}

\author{Rebeka Kiss}
\address{Bolyai Institute \\
	University of Szeged \\
	Aradi v\'ertan\'uk tere 1\\
	H-6720 Szeged, Hungary}
\email{Kiss.Rebeka@stud.u-szeged.hu}
\author{G\'abor P. Nagy}
\address{Department of Algebra \\
	Budapest University of Technology and Economics\\
	Egry J\'ozsef utca 1\\
	H-1111 Budapest, Hungary}
\address{Bolyai Institute \\
	University of Szeged \\
	Aradi v\'ertan\'uk tere 1\\
	H-6720 Szeged, Hungary}
\email{nagyg@math.bme.hu}

\thanks{Support provided from the National Research, Development and Innovation Fund of Hungary, financed under the 2018-1.2.1-NKP funding scheme, within the SETIT Project 2018-1.2.1-NKP-2018-00004. Partially supported by NKFIH-OTKA Grants 119687, 115288 and SNN 132625.}


\begin{document}

\begin{abstract}
We show that no orthogonal arrays $OA(16 \lambda, 11, 2,4)$ exist with $\lambda=6$ and $\lambda=7$. This solves an open problem of the NSUCRYPTO Olympiad 2018. Our result allows us to determine the minimum weights of certain higher-order correlation-immune Boolean functions.
\end{abstract}

\maketitle

\section{Introduction}

In the Fifth International Students' Olympiad in Cryptography NSUCRYPTO'2018 \cite{nsu18paper,problem18} the following problem was stated. Given three positive integers $n$, $t$, and $\lambda$ such that $t < n$, we call a $\lambda 2^t \times n$ binary array (i.e., matrix over the two-element field) a $t-(2, n, \lambda)$ orthogonal array if in every subset of $t$ columns of the array, every (binary) $t$-tuple appears in exactly $\lambda$ rows. $t$ is called the strength of this orthogonal array. Find a $4-(2, 11, \lambda)$ orthogonal array with minimal value of $\lambda$. So far, the best known answer to this question is $\lambda=8$. Delsarte's Linear Programming Bound \cite[Theorem 4.15 and Table 4.19]{hedayat} implies $\lambda \geq 6$. 

In this short note, we use the terminology of the monograph \cite{hedayat} and we denote a $t-(2, n, \lambda)$ orthogonal array by $OA(2^t \lambda, n,2,t)$. The integers $N=2^t \lambda$ and $n$ are called the number of runs and the number of factors of the array. In an orthogonal array, the same row can occur multiple times. The orthogonal array is \textit{simple,} if each row occurs exactly once. 

Our solution to the problem is stated in the following theorem.

\begin{theorem}\label{thm:1}
No orthogonal arrays $OA(16 \lambda, 11, 2,4)$ exist with $\lambda=6$ and $\lambda=7$. 
\end{theorem}

A Boolean function $f\colon\mathbb{F}_2^n\to \mathbb{F}_2$ is \textit{correlation-immune} of some order $t <n$ (in brief, $t$-CI) if fixing at most $t$ of the $n$ input variables $x_1,\ldots,x_n$ does not change the output distribution of the function, whatever are the positions chosen for the fixed variables and the values chosen for them. Equivalently, the support of the function must be a simple binary orthogonal array of strength $t$, see \cite{CarletChen,CarletGuilley}. The weight of a Boolean function is the size of its support. Low weight $t$-CI Boolean functions have practical importance in cryptography, since they resist the Siegenthaler attack. Furthermore, $t$-CI Boolean functions allow reducing the overhead while keeping the same resistance to side channel attacks. See \cite{CarletGuilley} and the references therein. 

Theorem \ref{thm:1} allows us to determine the minimum weights of $t$\textsuperscript{th}-order correlation-immune Boolean functions in $n$ variables,
\[n \in \{11,12,13\}, \quad t \in \{4,5\}.\]
These values were marked as unknown in \cite[Table 2]{CarletChen} and \cite[Table 2]{CarletGuilley}.

We would like to thank Claude Carlet (Paris, France and Bergen, Norway) for his detailed comments on the previous version of this paper.

\section{Proof of the theorem}

Our proof uses the results by Bulutoglu and Margot \cite{margot}, by Schoen, Eendebak and Nguyen \cite{eendebak,eendebak-corr} and also by Eendebak \cite{eendebakpage}, where orthogonal arrays with small parameters are classified. Bulutoglu and Margot \cite{margot} used integer linear programming (ILP) methods, while the algorithms of Schoen, Eendebak and Nguyen \cite{eendebak,eendebak-corr} are based on the systematic study of the extensions of orthogonal arrays by new columns. Moreover, both approaches must deal with the isomorphism problem of orthogonal arrays. 

\begin{proof}[Proof of the theorem]
	From \cite[Table 1]{margot}, \cite[Table III]{eendebak} and also \cite{eendebakpage} we can see that no $OA(96,8,2,4)$ and no $OA(112,7,2,4)$ exist. We explain the relevant rows of the two tables. In \cite[Table 1]{margot} there are 4 columns with the following meanings:
	\begin{itemize}
		\item OA gives the parameters of the classified orthogonal array.
		\item $m'$ is the number of linearly independent equality constraints of the generated ILP problem.
		\item $p_{\max}$ is an upper bound on the maximum number of times a run can appear in an $OA(2^t\lambda, n,2,t)$.
		\item $h$ is the number of non-isomorphic orthogonal arrays with the given parameters.
	\end{itemize}
\begin{table}[ht]
	\centering	
	\begin{tabular}{|c|c|c|c|}
		\hline
		OA & $m'$ & $p_{\max}$ & $h$ \\ \hline
		$OA(96,8,2,4)$ & $163$ & $2$ & $0$ \\ \hline
		$OA(112,7,2,4)$ & $99$ & $3$ & $0$ \\ \hline
	\end{tabular}
\end{table}

	From this table we can see that if $\lambda=6$ then no orthogonal array exists with $n=8$, which implies that no OA exists with $n\geq{8}$. Similarly if $\lambda=7$ then no orthogonal array exists with $n=7$, thus no OA exists with $n\geq{7}$.
	
	In \cite[Table III]{eendebak} orthogonal arrays with strength 4 are included, where 
	\begin{itemize}
		\item $N$ gives the run-size of the classified orthogonal array.
		\item The notation $2^a$ for the factor set means a binary array with $a$ factors. 
		\item $a_{\max}$ is the maximum number $a$, such that there exists an $OA$ with $N$ runs and $a$ factors. 
		\item The numbers $m_a$, $a \in \{t+1, \ldots,a_{\max}\}$, in the last column denote the number of isomorphism classes of arrays with $N$ runs and $a$ factors. 
	\end{itemize}
\begin{table}[ht]
	\centering	
	\begin{tabular}{|c|c|c|c|}
		\hline
		$N$ & Factor set & $a_{\max}$ & Isomorphism classes \\ \hline
		$96$ & $2^a$ & $7$ & $4,9,4$ \\ \hline
		$112$ & $2^a$ & $6$ & $4,3$ \\ \hline
	\end{tabular}
\end{table}

	This means that with run-size $96$ the maximum number $a$ such that an $OA(96,a,2,4)$ exists is $7$, and with run-size $112$ the maximum number $a$ with an existing $OA(112,a,2,4)$ is 6. 
\end{proof}

\begin{remark}
	According to \cite{eendebak-corr}, the number of isomorphism classes of binary orthogonal arrays with run-size $N=128$, factor-size $n=11$, and strength $t=4$ is $477$. The papers \cite{margot,eendebak} claim to achieve the above results within a few seconds. Using SageMath \cite{sagemath}, the GLPK package \cite{GLPK} and the integer linear programming solver SCIP \cite{GamrathEtal2020OO}, a straightforward implementation of the formulas of \cite{margot} used $51\,630$ sec and $481$ sec CPU time for the nonexistence of $OA(96,8,2,4)$ and $OA(112,7,2,4)$, respectively. 
\end{remark}

\section{Minimum weight of correlation-immune Boolean functions}

Using the notation of \cite{hedayat}, we denote by $F(n,2,t)$ the minimal number of runs $N$ in any $OA(N,n,2,t)$ for given values $n$ and $t$. Theorem \ref{thm:1} says that $F(11,2,4)\geq 128$, and in fact, equality holds. Let $\omega_{n,t}$ denote the minimum weight of $t$-CI Boolean functions in $n$ variables. Equivalently, $\omega_{n,t}$ is the minimum number of runs in a \textit{simple} orthogonal array with number of factors $n$ and strength $t$. Hence, 
\begin{align}
F(n,2,t)\leq \omega_{n,t}. \label{eq:1}
\end{align}
Suppose $A$ is an $OA(N,n,s,t)$. As in \cite[5]{hedayat}, one can construct an $OA(N/s,n-1,s,t-1)$, say $A'$. Clearly, if $A$ is simple then $A'$ is simple too. This implies
\begin{align}
F(n-1,2,t-1)&\leq \frac{1}{2}F(n,2,t), \label{eq:2}\\
\omega_{n-1,t-1} &\leq \frac{1}{2}\omega_{n,t}. \label{eq:3}
\end{align}

We are now able to fill some unknown values of \cite[Table 2]{CarletChen} and \cite[Table 2]{CarletGuilley}.

\begin{proposition}
For the minimum weight of $t$-CI Boolean functions in $n$ variables, we have
\begin{align} \label{eq:4}
\omega_{11,4}=\omega_{12,4}=\omega_{13,4}=\omega_{14,4}=\omega_{15,4}=128,
\end{align}
and
\begin{align} \label{eq:5}
\omega_{11,5}=\omega_{12,5}=\omega_{13,5}=\omega_{14,5}=\omega_{15,5}=\omega_{16,5}=256.
\end{align}
\end{proposition}
\begin{proof}
	The Nordstrom--Robinson code and also Sloane gives a simple $OA(256,16,2,5)$, see \cite{nsu18paper,BIERBRAUER2007158} and \cite{sloanepage}. Straightforward computation shows that deleting the last $5$ columns of it, the resulting orthogonal array is simple. Hence, $\omega_{n,5}\leq 256$ for $n\in \{11,\ldots,16\}$. By \eqref{eq:3}, $\omega_{n,4} \leq 128$ for $n\in \{10,\ldots,15\}$. Theorem \ref{thm:1} implies $F(n,2,4)\geq 128$ for $n\geq 11$. From \eqref{eq:1} and \eqref{eq:3} follow \eqref{eq:4} and \eqref{eq:5}. 
\end{proof}

\printbibliography

\end{document}